\documentclass[a4paper,11pt,reqno]{amsproc}
\usepackage{amssymb}
\usepackage[english]{babel} 
\usepackage[latin1]{inputenc}
\usepackage{enumerate}
\usepackage{fullpage}
\usepackage{lmodern}
\newtheorem{corollary}{Corollary}
\newtheorem{theorem}[corollary]{Theorem}
\newtheorem{fact}[corollary]{Fact}
\newtheorem{proposition}[corollary]{Proposition}
\newtheorem{lemma}[corollary]{Lemma}
\theoremstyle{definition}
\newtheorem{definition}[corollary]{Definition}
\newtheorem{remark}[corollary]{Remark}

\newcommand{\vek}[1]{\mathbf{#1}} 
\newcommand{\mat}[1]{\mathbf{#1}}
\newcommand{\abs}[1]{\lvert#1\rvert}
\newcommand{\whom}{\mathrm{w}_{\mathrm{hom}}}

\newcommand{\wham}{\mathrm{w}_{\mathrm{Ham}}}
\newcommand{\wlee}{\mathrm{w}_{\mathrm{Lee}}}

\newcommand{\tp}{\mathsf{T}}

\DeclareMathOperator{\Mat}{M}

\DeclareMathOperator{\rank}{rk} 
\DeclareMathOperator{\rad}{rad} 
\DeclareMathOperator{\soc}{soc}

\DeclareMathOperator{\Hom}{Hom}
\DeclareMathOperator{\PG}{PG}
\DeclareMathOperator{\GL}{GL}

\newcommand{\points}{\mathcal{P}}

\newcommand{\R}{\mathbb{R}} 
\newcommand{\Z}{\mathbb{Z}}
\newcommand{\N}{\mathbb{N}} 
\newcommand{\Q}{\mathbb{Q}}
\newcommand{\C}{\mathbb{C}}
\newcommand{\F}{\mathbb{F}}

\newcommand{\wrt}{w.\,r.\,t.}  
\newcommand{\ie}{i.\,e.}  

\newcommand{\oif}{\Longrightarrow}

\title[]{The Geometry of Homogeneous Two-Weight Codes}

\author{Thomas Honold}
\address{Thomas Honold\\
  Institute of Information and Communication Engineering\\
  Zhejiang University, Zheda Road\\
  310027 Hangzhou, China}
\email{honold@zju.edu.cn}
\thanks{Expanded version with proofs of \emph{Further Results on
    Homogeneous Two-Weight Codes}, which appeared in the OC2007
  conference proceedings.}

\date{}

\begin{document}

\begin{abstract}
  The results of \cite{emt:oc4graphs,emt:graphs} on linear homogeneous
  two-weight codes over finite Frobenius rings are exended in two
  ways: It is shown that certain non-projective two-weight codes give
  rise to strongly regular graphs in the way described in
  \cite{emt:oc4graphs,emt:graphs}. Secondly, these codes are used to
  define a dual two-weight code and strongly regular graph similar to
  the classical case of projective linear two-weight codes over finite
  fields \cite{calderbank-kantor86}.
\end{abstract}



\maketitle

\section{Introduction}\label{sec:intro}

A finite ring $R$ is said to be a Frobenius ring if there exists a character
$\chi\in\widehat{R}=\Hom_\Z(R,\C^\times)$ whose kernel contains no
nonzero left (or right) ideal of $R$.
The (normalized) homogeneous weight $\whom\colon R\to\C$ on a finite
Frobenius ring $R$ is defined by
\begin{equation}
  \label{eq:whom}
  \whom(x)=1-\frac{1}{\abs{R^\times}}\sum_{u\in\R^\times}\chi(ux).
\end{equation}
(This does not depend on the choice of $\chi$.)
The function $\whom$ is the unique complex-valued function on $R$ satisfying
$\whom(0)=0$, $\whom(ux)=\whom(x)$ for $x\in R$, $u\in R^\times$ and
$\sum_{x\in I}\whom(x)=\abs{I}$ for all nonzero left ideals
$I\leq{}_RR$ (and their right counterparts).

The homogeneous weight on a finite Frobenius ring is a generalization of both
the Hamming weight on $\F_q$ ($\whom(x)=\frac{q}{q-1}\wham(x)$ for
$x\in\F_q$) and the Lee weight on $\Z_4$ ($\whom(x)=\wlee(x)$ for
$x\in\Z_4$). It was introduced in
\cite{ioana-werner97} for the case $R=\Z_m$ and generalized to
Frobenius rings in \cite{greferath-schmidt99,st:homippi}.

In \cite{emt:oc4graphs,emt:graphs} it was shown that a linear code $C$
over a finite Frobenius ring with exactly two nonzero homogeneous weights
and satisfying certain nondegeneracy conditions gives rise to a
strongly regular graph with $C$ as its set of vertices. In the
classical case $R=\F_q$ this result has been known for a long time and
forms part of the more general correspondence between projective linear
$[n,k]$ two-weight codes and $\{\lambda_1,\lambda_2\}$
difference sets over $\F_q$ and their (appropriately defined) duals
(cf.\cite{calderbank-kantor86,delsarte72}).

In this paper we generalize the results of
\cite{emt:oc4graphs,emt:graphs} to a larger class of homogeneous
two-weight codes (so-called modular two-weight codes) and establish
for these codes the classical correspondence (Theorems~3.2 and~5.7 of
\cite{calderbank-kantor86}) in full generality.

\section{A few properties of Frobenius rings and their 
homogeneous  weights}\label{sec:properties} 

For a subset $S$ of a ring $R$
let ${}^\perp S=\{x\in R;xS=0\}$, $S^\perp=\{x\in R;Sx=0\}$. Similarly,
for $S\subseteq R^n$ let ${}^\perp S=\{\vek{x}\in R^n;\vek{x}\cdot S=0\}$ and
$S^\perp=\{\vek{x}\in R^n;S\cdot\vek{x}=0\}$, where
$\vek{x}\cdot\vek{y}=x_1y_1+\dots+x_ny_n$.

\begin{proposition}
  \label{prop:same-shape}
  A finite ring $R$ is a Frobenius ring iff for every matrix
  $\mat{A}\in R^{m\times n}$ the left row space
  $C=\{\vek{x}\mat{A};\vek{x}\in R^m\}$ and the right column space
  $D=\{\mat{A}\vek{y};\vek{y}\in R^n\}$ have the same cardinality.
\end{proposition}
\begin{proof}
  Suppose first that $R$ is Frobenius. The map $R^n\to R^m$,
  $\vek{y}\mapsto\mat{A}\vek{y}$ has kernel $C^\perp$ and image
  $D$. By the homomorphism theorem,
  $\abs{C^\perp}\abs{D}=\abs{R^n}$. Thus $\abs{C}=\abs{D}$ iff
  $\abs{C}\abs{C^\perp}=\abs{R^n}$. Since finite Frobenius rings are
  characterized by the property that $\abs{C}\abs{C^\perp}=\abs{R^n}$
  for every $n$ and every submodule $\abs{C}\leq{}_RR^n$ (cf.\
  for example \cite{t:fnote}), the result follows.  
\end{proof}

From now on we suppose that $R$ is a finite Frobenius ring with
homogeneous weight $\whom$. First we recall an alternative expression
for $\whom(x)$ derived in \cite{st:homippi}. Suppose $R/\rad
R=\prod_{i=1}^tR_i$ where $R_i\cong\Mat(m_i,\F_{q_i})$ is a simple
ring and $\abs{R_1}\leq\abs{R_2}\leq\dots\leq\abs{R_t}$. Then
$\soc(R)=\bigoplus_{i=1}^t S_i$ where $S_i\cong R_i$ (as an
$R_i$-$R_i$-bimodule)
is the two-sided ideal of $R$ defined by $R_iS_i=S_iR_i=S_i$ and
$R_iS_j=S_jR_i=\{0\}$ for $1\leq i,j\leq t$.
% \begin{fact}[cf.\ \cite{st:homippi}]
%   \label{fact:homippi}

From \cite{st:homippi} we have the following: 
If $x\notin\soc(R)$ then $\whom(x)=1$.  If $x\in\soc(R)$,
$x=\sum_{i=1}^tx_i$ with $x_i\in S_i$, then
\begin{equation}
  \label{eq:whomippi}
  \whom(x)=1-\prod_{i=1}^t\frac{(-1)^{\rank
      x_i}}{(q_i^{m_i}-1)(q_i^{m_i-1}-1)
    \dotsm(q_i^{m_i-\rank x_i+1}-1)},
\end{equation}
where $\rank\colon S_i\to\N_0$ denotes the ``matrix rank'' induced by
the isomorphism $S_i\cong\Mat(m_i,\F_{q_i})$.
%\end{fact}

Next we determine the set of all $x\in R$ satisfying $\whom(x)=0$. Let
$S_i=Rs_i$, $1\leq i\leq\tau$, be the different left ideals of $R$ of
order $2$ and $S=S_1+\dots+S_\tau$. The set $S$ is a two-sided ideal
of $R$ of order $2^\tau$, whose elements are the subset sums of
$\{s_1,\dots,s_\tau\}$.
% $R/\rad R\cong\prod_{i=1}^t\Mat_{m_i}(\F_{q_i})$. 
Define $S_0\subseteq S$ as the set of all sums of an even number of
elements from $\{s_1,\dots,s_\tau\}$ (``even-weight subcode of $S$'').
Note that $S_0$ is a subgroup of $(R,+)$, trivial for
$\tau\leq 1$ and nontrivial (of order $2^{\tau-1}$) for $\tau\geq 2$.
\begin{proposition}
  \label{prop:S_0}
  We have $\whom(x)\geq 0$ for all $x\in R$ and $S_0=\{x\in
  R;\whom(x)=0\}$. Moreover, $\whom(x+y)=\whom(x)$ for all $x\in R$
  and $y\in S_0$.
\end{proposition}
\begin{proof}
  This follows from a close inspection of the formula
  \eqref{eq:whomippi}.
\end{proof}
\begin{fact}[cf.~\cite{wt:egal}]
  \label{fact:egal}
\begin{equation}
  \label{eq:egal}
  \sum_{x\in I}\whom(x+c)=\abs{I}
\end{equation}
for all nonzero left (or right) ideals $I$ of $R$ and all $c\in
R$. 
\end{fact}
The following correlation property of $\whom$ turns out to be crucial.
\begin{proposition}
  \label{prop:whom-corr}
  For a nonzero left ideal $I$ of $R$ and $r,s\in R$ we have
  \begin{equation}
    \label{eq:whom-corr}
    \sum_{x\in I}\whom(x)\whom(xr+s)
    =
    \begin{cases}
\abs{I}+\abs{I}\cdot\frac{\abs{R^\times\cap(1+I^\perp)}}
    {\abs{R^\times}}\cdot\bigl(1-\whom(s)\bigr)&\text{if $\abs{Ir}=\abs{I}$},\\
    \abs{I}&\text{if $\abs{Ir}<\abs{I}$}.
    \end{cases}
 \end{equation}
  In particular $\sum_{x\in
    R}\whom(x)^2=\abs{R}\cdot\left(1+\frac{1}{\abs{R^\times}}\right)$.
\end{proposition}
\begin{proof}
  Denote the left-hand side of \eqref{eq:whom-corr} by $\rho(s)$. Using
  \eqref{eq:egal} it
  is easily verified that $\rho(us)=\rho(s)$ for $s\in R$, $u\in
  R^\times$ and $\sum_{s\in J}\rho(s)=\abs{I}\abs{J}$ for all nonzero
  left ideals $J\leq{}_RR$. Hence
  $\rho(s)=\rho(0)+\bigl(\abs{I}-\rho(0)\bigr)\whom(s)$ for $s\in R$.

  If $\abs{Ir}<\abs{I}$ then $K:=I\cap{}^\perp r\neq 0$. Hence 
  choosing $x_a\in I$ with $x_ar=a$ (for $a\in Ir$) we get
  \begin{align*}
    \rho(0)&=\sum_{x\in I}\whom(x)\whom(xr)=\sum_{a\in Ir}\left(\sum_{x\in
      K+x_a}\whom(x)\right)\whom(a)\\
  &=\abs{K}\sum_{a\in Ir}\whom(a)=\abs{K}\abs{Ir}=\abs{I}.    
  \end{align*}

  If $\abs{Ir}=\abs{I}$, \ie\ $I\to Ir$, $x\mapsto xr$ is an
  isomorphism of left $R$-modules, then $\whom(xr)=\whom(x)$ and hence
  \begin{align*}
    \rho(0)=\sum_{x\in I}\whom(x)^2&=\sum_{x\in
      I}\left(1-\frac{1}{\abs{R^\times
        }}\sum_{u\in R^\times}\chi(xu)\right)^2\\
    &=\abs{I}+\frac{1}{\abs{R^\times}^2}\cdot\sum_{u,v\in R^\times}
    \sum_{x\in I}\chi\bigl(x(u+v)\bigr)\\
    &=\abs{I}+\frac{\abs{I}}{\abs{R^\times}^2}\cdot
    \abs{\bigl\{(u,v)\in R^\times\times R^\times;I(u+v)=0\bigr\}}\\
    &=\abs{I}+\frac{\abs{I}}{\abs{R^\times}}\cdot
    \abs{\bigl\{u\in R^\times;I(u-1)=0\bigr\}}\\
    &=\abs{I}+\frac{\abs{I}}{\abs{R^\times}}
    \cdot\abs{R^\times\cap(1+I^\perp)}.
  \end{align*}
\end{proof}
For vectors $\vek{x},\vek{y}\in R^k$ we write $\vek{x}\sim\vek{y}$ if
$\vek{x}R^\times=\vek{y}R^\times$. By \cite[Prop.~5.1]{wood99a}
this is equivalent to $\vek{x}R=\vek{y}R$.
\begin{proposition}
  \label{prop:whom-corr-vek}
  For nonzero words $\vek{g},\vek{h}\in R^k$ and $s\in R$ we have
  \begin{equation}
    \label{eq:whom-corr-vek}
    \sum_{\vek{x}\in R^k}\whom(\vek{x}\cdot\vek{g})
    \whom(\vek{x}\cdot\vek{h}+s)=
    \begin{cases}
      \abs{R}^k+\frac{\abs{R}^k}{\abs{\vek{g}R^\times}}
      \cdot\bigl(1-\whom(s)\bigr) 
      &\text{if $\vek{g}\sim\vek{h}$},\\
      \abs{R}^k
      &\text{if $\vek{g}\nsim\vek{h}$}.
    \end{cases}
  \end{equation}
\end{proposition}
\begin{proof}
  Reasoning as in the proof of Prop.~\ref{prop:whom-corr} the
  left-hand side $\rho(s)$ of \eqref{eq:whom-corr-vek} satisfies
  $\rho(s)=\rho(0)+\bigl(\abs{R}^k-\rho(0)\bigr)\whom(s)$ for $s\in R$.

  For $\vek{g}\in R^k$, $a\in R$ the equation
  $\vek{x}\cdot\vek{g}=x_1g_1+\dots+x_kg_k=a$ is solvable if and only
  if $a\in Rg_1+\dots+Rg_k$. If this is true and
  $\vek{x}_a$ denotes a particular solution, the set of all solutions is
  the coset $\vek{x}_a+{}^\perp\vek{g}$. Hence
  \begin{equation}
    \label{eq:whom-corr-vek-p1}
    \rho(0)=\sum_{\vek{x}\in R^k}\whom(\vek{x}\cdot\vek{g})
    \whom(\vek{x}\cdot\vek{h})=
    \sum_{a\in
      Rg_1+\dots+Rg_k}\whom(a)\sum_{\vek{y}\in{}^\perp\vek{g}}
    \whom\bigl((\vek{x}_a+\vek{y})\cdot\vek{h}\bigr).
  \end{equation}
  There are now two cases to consider.
  
  Case~(i): $\vek{h}\in\vek{g}R=({}^\perp\vek{g})^\perp$. Letting
  $\vek{h}=\vek{g}r$ we have $(\vek{x}_a+\vek{y})\cdot\vek{h}
  =(\vek{x}_a\cdot\vek{g})r+(\vek{y}\cdot\vek{g})r=ar$ for all
  $\vek{y}\in{}^\perp\vek{g}$ and hence
  \begin{equation*}
    \rho(0)=
    \abs{{}^\perp\vek{g}}\cdot\sum_{a\in Rg_1+\dots+Rg_k}\whom(a)\whom(ar).
  \end{equation*}
If $\vek{h}R=\vek{g}R$, we may assume $r\in R^\times$. Applying
Proposition~\ref{prop:whom-corr}
to $I=Rg_1+\dots+Rg_k$ and using
$\abs{{}^\perp\vek{g}}\abs{I}
=\abs{{}^\perp\vek{g}}\abs{Rg_1+\dots+Rg_k}=\abs{R}^k$, $I^\perp=\{r\in
R;g_1r=\dots=g_kr=0\}=\{r\in R;\vek{g}r=\vek{0}\}$,
$R^\times\cap(1+I^\perp)=R^\times\cap(1+\vek{g}^\perp)
=\{u\in R^\times;\vek{g}u=\vek{g}\}$ gives
\begin{align*}
  \rho(0)&=\abs{{}^\perp\vek{g}}\cdot\sum_{a\in I}\whom(a)^2
  =\abs{R}^k+\abs{R}^k\cdot\frac{\abs{\{u\in
        R^\times;\vek{g}u=\vek{g}\}}}{\abs{R^\times}}\\
  &=\abs{R}^k+\frac{\abs{R}^k}{\abs{\vek{g}R^\times}}
\end{align*}
as desired. Otherwise $\vek{h}R=\vek{g}rR\subsetneq\vek{g}R$,
${}^\perp\vek{g}\subsetneq{}^\perp(\vek{g}r)$ and hence there exists
$\vek{x}\in R^k$ such that $a:=x_1g_1+\dots x_kg_k\neq 0$,
$ar=x_1g_1r+\dots x_kg_kr=0$. Thus $\abs{Ir}<\abs{I}$, and
Proposition~\ref{prop:whom-corr} then implies
\begin{equation*}
  \rho(0)=\abs{{}^\perp\vek{g}}\cdot\abs{I}=\abs{R}^k.
\end{equation*}

Case~(ii): $\vek{h}\notin\vek{g}R$. Here
$\vek{x}_a\cdot\vek{h}+{}^\perp\vek{g}\cdot\vek{h}$ is a coset of a
nonzero left ideal of $R$ and hence
\begin{equation}
  \label{eq:whom-corr-vek-p3}
  \rho(0)=\abs{{}^\perp\vek{g}}\cdot\sum_{a\in Rg_1+\dots+Rg_k}\whom(a)
  =\abs{{}^\perp\vek{g}}\cdot\abs{Rg_1+\dots+Rg_k}=\abs{R^k},
\end{equation}
completing the proof of Prop.~\ref{prop:whom-corr-vek}.
\end{proof}

\section{Modular Two-Weight Codes, Partial Difference Sets and Strongly
  Regular Graphs}

Given a positive integer $k$, the set of nonzero cyclic submodules of the free
right module $R^k_R$ is denoted by $\points$. The elements of
$\points$ are referred to as \emph{points} of the projective geometry
$\PG(R^k_R)$, and a multiset $\alpha\colon\points\to\N_0$ is referred
to as a \emph{multiset in $\PG(R^k_R)$}.

With a left linear code $C\leq{}_RR^n$ generated by $k$ (or fewer)
codewords and having no all-zero coordinate we
associate a multiset $\alpha_C$ in $\PG(R^k_R)$ of cardinality $n$ in
the following way: If $C=\{\vek{x}\mat{G};\vek{x}\in R^k\}$ with
$\mat{G}=(\vek{g}_1|\vek{g}_2|\dots|\vek{g}_n)\in R^{k\times n}$,
define $\alpha_C\colon\points\to\N_0$ by
$\alpha(\vek{g}R)=\abs{\{j;\vek{g}_jR=\vek{g}R\}}$. The relation
$C\leftrightarrow\alpha_C$ defines a bijection between classes of
monomially isomorphic left linear codes over $R$ generated by $k$
codewords and orbits of the group $\GL(k,R)$ on multisets in
$\PG(R^k_R)$.

\begin{definition}
  \label{dfn:modular}
%   A multiset $\alpha$ in $\PG(R^k_R)$ is said to be \emph{modular} if
%   there exists $r\in\Q$ such that for all points $\vek{g}R$ of
%   $\PG(R^k_R)$ either $\alpha(\vek{g}R)=0$ or
%   $\alpha(\vek{g}R)=r\abs{\vek{g}R^\times}$.  A code $C\leq{}_RR^n$ is
%   said to be \emph{modular} if it is associated with a modular
%   multiset $\alpha_C$ in $\PG(R^k_R)$.  
  
  A code $C\leq{}_RR^n$ is said to be \emph{modular} if there exists
  $r\in\Q$ such that for all points $\vek{g}R$ of $\PG(R^k_R)$ either
  $\alpha_C(\vek{g}R)=0$ or
  $\alpha_C(\vek{g}R)=r\abs{\vek{g}R^\times}$. The number $r$ is
  called the \emph{index} of $C$.
\end{definition}

% The modular multisets in $\PG(R^k_R)$ are exactly those multisets
% which are derived from a subset $A\subseteq R^k\setminus\{\vek{0}\}$
% satisfying $AR^\times=A$ in the obvious way, \ie\
% $\alpha(\vek{g}R)=\abs{\{\vek{g}'\in A;\vek{g}'R=\vek{g}R\}}$, and
% their rational multiples.
The property of $C$ described in Def.~\ref{dfn:modular} does not
depend on the choice of $\alpha_C$ (not even on the dimension $k$).
Hence modularity of a linear code is a well-defined concept.

If $A\subseteq R^k\setminus\{\vek{0}\}$ satisfies $AR^\times=A$, the
matrix $\mat{G}$ with the vectors of $A$ as columns generates a
modular (left) linear code of length $\abs{A}$ and index $1$.

% Up to (right) scaling of the columns by elements of $R^\times$ and
% $r$-fold repetition of the columns (where $R$ may be a proper
% fraction!), every modular linear code arises in this way. 

Note that projective codes over $\F_q$ are
modular of index $\frac{1}{q-1}$ and regular
projective codes over $R$ as defined in \cite{emt:oc4graphs,emt:graphs} are
modular of index $\frac{1}{\abs{R^\times}}$.

\begin{fact}[{\cite[Th.~5.4]{wood02}}]
  \label{fact:equidistant}
  A linear code $C\leq{}_RR^n$ is a one-weight code (\ie\ equidistant
  \wrt\ $\whom$)
  iff $C$ is modular and
  $\bigl\{\vek{g}\in R^k\setminus\{\vek{0}\};\alpha_C(\vek{g}R)>0\bigr\}$
  is the set of nonzero vectors of a submodule of $R^k_R$.
\end{fact}

The main purpose of this paper is a combinatorial characterization
of (linear) homogeneous two-weight codes over $R$, \ie\ codes over $R$
having exactly two nonzero homogeneous weights $w_1<w_2$. Assuming that $C$ is
such a code, we set $w_0=0$, $C_i=\{\vek{c}\in C;\whom(\vek{c})=w_i\}$ and
$b_i=\abs{C_i}$ for $i=0,1,2$.

By Prop.~\ref{prop:S_0} we have $\whom(\vek{c})=0$ iff $\whom(c_j)=0$ for
$1\leq j\leq n$, the set $C_0$ is a subgroup of $(C,+)$ and $C_1$, $C_2$ are
  unions of cosets of $C_0$. If the weights $w_1$, $w_2$ and
  $b_0=\abs{C_0}$ are known, the frequencies $b_1$, $b_2$ can be
  computed from the equations $b_1+b_2=\abs{C}-\abs{C_0}$,
  $b_1w_1+b_2w_2=\sum_{\vek{c}\in C}\whom(\vek{c})=n\abs{C}$ (assuming
  that $C$ has no all-zero coordinate) and are given by
  \begin{equation}
    \label{eq:b_i}
    b_1=\frac{(w_2-n)\abs{C}-w_2\abs{C_0}}{w_2-w_1},\quad
    b_2=\frac{(n-w_1)\abs{C}+w_1\abs{C_0}}{w_2-w_1}.
  \end{equation}
\begin{lemma}
  \label{lma:whom-corr-code}
  For a modular code $C\leq{}_RR^n$ of index $r$ and $\vek{d}\in R^n$ we have
  \begin{equation*}
    \sum_{\vek{c}\in C}\whom(\vek{c})\whom(\vek{c}+\vek{d})
    =\abs{C}\cdot\bigl(n^2+rn-r\cdot\whom(\vek{d})\bigr).
  \end{equation*}
\end{lemma}
\begin{proof}
Suppose that $C$ is generated by
$\mat{G}=(\vek{g}_1|\dots|\vek{g}_n)\in R^{k\times n}$. By assumption,
the multiset $\alpha$ in $\PG(R^k_R)$ obtained from $\mat{G}$
satisfies $\alpha(\vek{g}_jR)=r\abs{\vek{g}_jR^\times}$ for $1\leq
j\leq n$. We obtain
  \begin{align*}
    \frac{1}{\abs{C}}\sum_{\vek{c}\in
      C}\whom(\vek{c})\whom(\vek{c}+\vek{d})
    &=\frac{1}{\abs{R}^k}\sum_{\vek{x}\in
      R^k}\whom(\vek{x}\mat{G})\whom(\vek{x}\mat{G}+\vek{d})\\
    &=\frac{1}{\abs{R}^k}\sum_{i,j=1}^n\sum_{\vek{x}\in
      R^k}\whom(\vek{x}\cdot\vek{g}_i)\whom(\vek{x}\cdot\vek{g}_j+d_j)\\
    &=n^2+\sum_{j=1}^n\frac{1-\whom(d_j)}{\abs{\vek{g}_jR^\times}}
      \times\#\{i;\vek{g}_i\sim\vek{g}_j\}\tag{by
        Prop.~\ref{prop:whom-corr-vek}}\\
      &=n^2+\sum_{j=1}^n\frac{1-\whom(d_j)}{\abs{\vek{g}_jR^\times}}
        \times\alpha(\vek{g}_jR)\\
        &=n^2+rn-r\cdot\whom(\vek{d})
  \end{align*}
  as asserted.
\end{proof}
In the special case $\vek{d}=\vek{0}$ Lemma~\ref{lma:whom-corr-code}
reduces to
$\sum_{\vek{c}\in C}\whom(\vek{c})^2=(n^2+rn)\abs{C}$.
\begin{lemma}
   The nonzero weights $w_1,w_2$ of a modular two-weight code
   $C\leq{}_RR^n$ of index $r$ satisfy the relation
   \begin{equation*}
     (w_1+w_2)n\abs{C}=(n^2+rn)\abs{C}+w_1w_2\bigl(\abs{C}-\abs{C_0}\bigr).
   \end{equation*}
\end{lemma}
\begin{proof}
  We have the system 
  \begin{equation*}
    \begin{array}{lclcl}
      b_1&+&b_2&=&\abs{C}-\abs{C_0},\\
      b_1w_1&+&b_2w_2&=&n\abs{C},\\
      b_1w_1^2&+&b_2w_2^2&=&(n^2+rn)\abs{C},
    \end{array}
  \end{equation*}
from which we obtain the asserted formula using
\begin{equation*}
  (w_1+w_2)(b_1w_1+b_2w_2)=(b_1w_1^2+b_2w_2^2)+w_1w_2(b_1+b_2).
\end{equation*}
\end{proof}
\begin{lemma}
  \label{lma:whom-coset-code}
  For a modular two-weight code $C\leq{}_RR^n$ of index $r$ and
  $\vek{d}\in R^n$ we have
  \begin{equation}
    \label{eq:whom-coset-code}
    \sum_{\vek{c}\in C_1}\whom(\vek{c}+\vek{d})=
    b_1w_1+\left(b_1-\frac{b_1w_1}{n}\right)\whom(\vek{d})
  \end{equation}
\end{lemma}
\begin{proof}
  Using Lemma~\ref{lma:whom-corr-code} we can setup the following
  system of equations for the unknowns $\sum_{\vek{c}\in
    C_1}\whom(\vek{c}+\vek{d})$ and $\sum_{\vek{c}\in
    C_2}\whom(\vek{c}+\vek{d})$:
  \begin{align*}
    \sum_{\vek{c}\in C_1}\whom(\vek{c}+\vek{d})
    +\sum_{\vek{c}\in C_2}\whom(\vek{c}+\vek{d})
    &=\sum_{\vek{c}\in C}\whom(\vek{c}+\vek{d})-\sum_{\vek{c}\in
      C_0}\whom(\vek{c}+\vek{d})\\
    &=n\abs{C}-\abs{C_0}\whom(\vek{d}),\\
    w_1\sum_{\vek{c}\in C_1}\whom(\vek{c}+\vek{d})
    +w_2\sum_{\vek{c}\in C_2}\whom(\vek{c}+\vek{d})
    &=\sum_{\vek{c}\in
      C}\whom(\vek{c})\whom(\vek{c}+\vek{d})\\
    &=\abs{C}\cdot\bigl(n^2+rn-r\cdot\whom(\vek{d})\bigr)
  \end{align*}
  Solving this system yields
  \begin{equation}
    \label{eq:whom-coset-code-p1}
    \sum_{\vek{c}\in C_1}\whom(\vek{c}+\vek{d})=
    \frac{\bigl(w_2n-n^2-rn)\abs{C}}{w_2-w_1}
    +\frac{r\abs{C}-w_2\abs{C_0}}{w_2-w_1}\cdot\whom(\vek{d})
  \end{equation}
for any $\vek{d}\in R^n$. The first summand in
\eqref{eq:whom-coset-code-p1} is equal to $b_1w_1$, as follows by
inserting $\vek{d}=\vek{0}$. This in turn gives
\begin{equation*}
  b_1-\frac{b_1w_1}{n}
  =\frac{(w_2-n)\abs{C}-w_2\abs{C_0}}{w_2-w_1}
  -\frac{(w_2-n-r)\abs{C}}{w_2-w_1}=\frac{r\abs{C}-w_2\abs{C_0}}{w_2-w_1},
\end{equation*}
transforming \eqref{eq:whom-coset-code-p1} into
\eqref{eq:whom-coset-code}.
\end{proof}
\begin{remark}
  \label{rmk:whom-code}
  Lemmas~\ref{lma:whom-corr-code} and ~\ref{lma:whom-coset-code} can
  be generalized to
  \begin{align*}
    \sum_{\vek{c}\in C}\whom(\vek{c})\whom(c_j+d_j)
    &=\abs{C}\cdot\bigl(n+r-r\cdot\whom(d_j)\bigr)\quad\text{and}\\
    \sum_{\vek{c}\in C_1}\whom(c_j+d_j)&=
    \frac{b_1w_1}{n}+\left(b_1-\frac{b_1w_1}{n}\right)\whom(d_j)
  \end{align*}
  respectively, where $j$ is any coordinate of $R^n$ and $d_j\in R$.
  In particular $\sum_{\vek{c}\in
    C_1}\whom(c_j)=\frac{b_1w_1}{n}$ is independent
  of $j$.
\end{remark}
Recall that a (simple)
graph $\Gamma$ is \emph{strongly regular with parameters
$(N,K,\lambda,\mu)$} if $\Gamma$ has $N$ vertices, is regular of degree $K$
and any two adjacent (resp.\ nonadjacent) vertices have $\lambda$
(resp.\ $\mu$) common neighbours. The graph $\Gamma$ is called \emph{trivial}
if $\Gamma$ or its complement is a disjoint unions of cliques of the same
size. This is equivalent to $\mu=0$ resp.\ $\mu=K$.

A subset $D\subset G$ of an (additively written) abelian group $G$ is
said to be a \emph{regular $(N,K,\lambda,\mu)$ partial difference set
  in $G$} if $N=\abs{G}$, $K=\abs{D}$, $0\notin D$, $-D=D$, and the
multiset $D-D$ represents each element of $D$ exactly $\lambda$ times and
each element of $G\setminus\bigl(D\cup\{0\}\bigr)$ exactly $\mu$ times;
cf.~\cite{ma94}.

If $D$ is a regular $(N,K,\lambda,\mu)$ partial difference set
in $G$, then the graph $\Gamma(G,D)$ with vertex set $G$ and edge set
$\bigl\{\{x,x+d\};x\in G,d\in D\bigr\}$ (the \emph{Cayley graph} of
$G$ \wrt\ $D$) is strongly regular with parameters
$(N,K,\lambda,\mu)$.

We are now ready to generalize the main result of
\cite{emt:graphs,emt:oc4graphs} to modular two-weight codes. For a
two-weight code $C$ we denote the Cayley graph $\Gamma(C/C_0,C_1/C_0)$
by $\Gamma(C)$. Thus the vertices of $\Gamma(C)$ are the cosets of $C_0$
in $C$, and two cosets $\vek{c}+C_0$, $\vek{d}+C_0$ are
adjacent iff $\whom(\vek{c}-\vek{d})=w_1$. As we have already mentioned,
Prop.~\ref{prop:S_0} ensures that $\Gamma(C)$ is well-defined.

% Given a two-weight code $C\leq{}_RR^n$ with nonzero weights $w_1<w_2$
% and $C_0=\bigl\{\vek{c}\in C;\whom(\vek{c})=0\bigr\}$,
% we define a graph $\Gamma(C)=(V,E)$ with vertex set $V=C/C_0$
% and edge set $E$ consisting of all unordered pairs
% $\vek{c}+C_0,\vek{d}+C_0$ with
% $\whom(\vek{c}-\vek{d})=w_1$. By Prop.~\ref{prop:C_0}, $\Gamma(C)$ is
% well-defined. The graph $\Gamma(C)$ is the Cayley graph of the abelian
% group $(C/C_0,+)$ \wrt\ the subset $C_1/C_0$.
\begin{theorem}
  \label{thm:Gamma(C)}
  The graph $\Gamma(C)$ associated with a modular two-weight code over
  a finite Frobenius ring $R$ is strongly regular with parameters
  \begin{gather*}
    N=\frac{\abs{C}}{\abs{C_0}},\quad
    K=\frac{(w_2-n)N-w_2}{w_2-w_1},\\
    \lambda=\frac{K\left(\frac{w_1^2}{n}-2w_1\right)+w_2(K-1)}{w_2-w_1},\quad
    \mu=\frac{K\left(\frac{w_1w_2}{n}-w_1-w_2\right)+w_2K}{w_2-w_1}.
  \end{gather*}
  The graph $\Gamma(C)$ is trivial iff $w_1=n$.
\end{theorem}
\begin{proof}
The numbers $b_i=\abs{C_i}$ satisfy the system of equations
\begin{equation*}
  \begin{array}{rcrcl}
    b_1&+&b_2&=&\abs{C}-\abs{C_0},\\
    w_1b_1&+&w_2b_2&=&\sum_{\vek{c}\in C}\whom(\vek{c})=n\abs{C},
  \end{array}
\end{equation*}
giving
\begin{equation*}
  \begin{pmatrix}
    b_1\\b_2
  \end{pmatrix}=\frac{1}{w_2-w_1}
  \begin{pmatrix}
    (w_2-n)\abs{C}-w_2\abs{C_0}\\
    (n-w_1)\abs{C}+w_1\abs{C_0}
  \end{pmatrix}\quad\text{and}\quad
  K=b_1/b_0=\frac{(w_2-n)N-w_2}{w_2-w_1}.
\end{equation*}
Let now
$b_{ij}(\vek{d})=\abs{C_i\cap(C_j+\vek{d})}
=\abs{\bigl\{\vek{c}\in C_i;\whom(\vek{c}-\vek{d})
  =w_j\bigr\}}$ for $\vek{d}\in C$ and $0\leq i,j\leq 2$. Our task
is to show that $b_{11}(\vek{d})$ depends only on the class $C_i$
containing $\vek{d}$. The cases $i=0$, $j=0$ and $\vek{d}\in C_0$ are
trivial. The numbers $b_{11}(\vek{d})$, $b_{12}(\vek{d})$ are the
solutions of the system of equations
\begin{align*}
    b_{11}(\vek{d})+b_{12}(\vek{d})&=\abs{C_1}-\abs{C_1\cap(C_0+\vek{d})}
      =
      \begin{cases}
        \abs{C_1}-\abs{C_0}&\text{if $\vek{d}\in C_1$},\\
        \abs{C_1}&\text{if $\vek{d}\in C_2$},
      \end{cases}\\
    w_1\cdot b_{11}(\vek{d})+w_2\cdot b_{12}(\vek{d})&=\sum_{\vek{c}\in C_1}
    \whom(\vek{c}-\vek{d}),
\end{align*}
whose coefficient matrix depends only on $\whom(\vek{d})$; cf.\
Lemma~\ref{lma:whom-coset-code}. This proves already that $\Gamma(C)$
is strongly regular with parameters
$\lambda=b_{11}(\vek{d})/\abs{C_0}$ and $\mu=b_{11}(\vek{d}')/\abs{C_0}$,
where $\vek{d}\in C_1$, $\vek{d}'\in C_2$. Using
Lemma~\ref{lma:whom-coset-code} we find
\begin{align*}
  \lambda&=\frac{w_2(b_1-b_0)-b_1w_1
    -\left(b_1-\frac{b_1w_1}{n}\right)w_1}{(w_2-w_1)b_0}
  =\frac{w_2(b_1/b_0-1)+(b_1/b_0)\left(\frac{w_1^2}{n}-2w_1\right)}{w_2-w_1}\\
  &=\frac{w_2(K-1)+K\left(\frac{w_1^2}{n}-2w_1\right)}{w_2-w_1},\\
  \mu&=\frac{w_2b_1-b_1w_1
    -\left(b_1-\frac{b_1w_1}{n}\right)w_2}{(w_2-w_1)b_0}
  =\frac{w_2(b_1/b_0)+(b_1/b_0)\left(\frac{w_1w_2}{n}-w_1-w_2
    \right)}{w_2-w_1}\\
  &=\frac{w_2K+K\left(\frac{w_1w_2}{n}-w_1-w_2\right)}{w_2-w_1}.
\end{align*}
Writing $\mu$ in the form
\begin{equation*}
  \mu=\frac{K\left(\frac{w_1w_2}{n}-w_1\right)}{w_2-w_1}
\end{equation*}
we see that $\mu=0$ ($\mu=K$) is equivalent to $w_2=n$ (resp.\
$w_1=n$). But
$n\abs{C}=b_1w_1+b_2w_2<\bigl(\abs{C}-\abs{C_0}\bigr)w_2$, so
$w_2>\frac{n\abs{C}}{\abs{C}-\abs{C_0}}>n$. Hence $\Gamma(C)$ is
trivial iff $\mu=K$ iff $w_1=n$.
\end{proof}
\begin{remark}
  \label{rmk:trivial}
  Since $\Gamma(C)$ is a Cayley graph,
  the preceding argument shows that $\Gamma(C)$ is trivial iff the
  codewords of weight $0$ and $w_2$ form a linear subcode of $C$ (and
  the cocliques of $\Gamma(C)$ are the cosets of $(C_0+C_2)/C_0$ in
  this case).
\end{remark}

\section{The Dual of a Modular Two-Weight Code}\label{sec:dual}

Suppose $C\leq{}_RR^n$ is a two-weight code over a finite Frobenius
ring with nonzero
weights $w_1$, $w_2$ and frequencies $b_1$, $b_2$. Let $\mat{M}_i\in
R^{b_i\times n}$ ($i=1,2$) be matrices whose rows are the 
codewords of $C$ of weight $w_i$ in some order.
\begin{definition}
  \label{dfn:dual}
  The right linear code $C'\leq R^{b_1}_R$ generated by the columns of
  $\mat{M}_1$ is called the \emph{dual of the two-weight code $C$}.
\end{definition}
The code $C'$ is modular of index $1$ (no matter whether $C$ is modular or
not).

\begin{theorem}
  \label{thm:dual}
  If $C\leq{}_RR^n$ is a modular two-weight code with $C_0=\{\vek{0}\}$, its
  dual $C'$ is also a (modular) two-weight code with
  $C'_0=\{\vek{0}\}$ and nonzero weights
  \begin{equation}
    \label{eq:dual}
    w_1'=\frac{(w_2-n-r)\abs{C}}{w_2-w_1}=\frac{b_1w_1}{n},\quad 
    w_2'=\frac{(w_2-n)\abs{C}}{w_2-w_1}.
  \end{equation}
%   The associated strongly regular graph $\Gamma(C')$ has parameters
%   \begin{equation}
%     \label{eq:Gamma(C')}
%     N'=\abs{C},\quad
%     K'=\frac{n}{r},\quad\lambda'=\frac{2n-w_1-w_2}{r}+\frac{w_1w_2}{\abs{C}},
%     \quad\mu'=\frac{w_1w_2}{\abs{C}}.
%   \end{equation}
%   The graph $\Gamma(C')$ is trivial iff $w_1=n$.
\end{theorem}
\begin{proof}
  Suppose $C$ is generated by $\mat{G}=(\vek{g}_1|\dots|\vek{g}_n)\in
  R^{k\times n}$. The weight of
  $\vek{d}_i=\mat{M}_i\vek{y}^\tp$ ($\vek{y}\in R^n$)
  is $\whom(\vek{d}_i)=\sum_{c\in C_i}\whom(\vek{c}\cdot\vek{y})$. By
  assumption, $C=C_1\uplus C_2\uplus\{\vek{0\}}$. Hence we get the
  following system of equations for $\whom(\vek{d}_1)$ and
  $\whom(\vek{d}_2)$:
  \begin{equation}
    \label{eq:dual-p1}
  \begin{aligned}
    \whom(\vek{d}_1)+\whom(\vek{d}_2)&=\sum_{\vek{c}\in C}
      \whom(\vek{c}\cdot\vek{y}),\\
      w_1\cdot\whom(\vek{d}_1)+w_2\cdot\whom(\vek{d}_2)&=\sum_{\vek{c}\in C}
      \whom(\vek{c})\whom(\vek{c}\cdot\vek{y}).
  \end{aligned}    
  \end{equation}
If $\vek{y}\in C^\perp$ then $\vek{d}_1=\vek{d}_2=0$. Otherwise
\begin{align*}
  \sum_{\vek{c}\in C}
      \whom(\vek{c}\cdot\vek{y})&=\abs{C},\\
      \sum_{\vek{c}\in C}
      \whom(\vek{c})\whom(\vek{c}\cdot\vek{y})
      &=\sum_{j=1}^n\sum_{\vek{c}\in C}
      \whom(c_j)\whom(\vek{c}\cdot\vek{y})\\
      &=\frac{\abs{C}}{\abs{R}^k}
      \sum_{j=1}^n\sum_{\vek{x}\in R^k}\whom(\vek{x}\cdot\vek{g}_j)
      \whom(\vek{x}\mat{G}\vek{y}^\tp)\\
      &=\abs{C}\left(n+\frac{\alpha(\mat{G}\vek{y}^\tp R^\times)}
        {\abs{\mat{G}\vek{y}^\tp R^\times}}\right)\\
      &=
      \begin{cases}
        (n+r)\abs{C}&\text{if $\mat{G}\vek{y}^\tp\sim\vek{g}_j$ for
          some $j$},\\
        n\abs{C}&\text{otherwise}.
      \end{cases}
\end{align*}
\end{proof}
Solving \eqref{eq:dual-p1} for $\whom(\vek{d}_1)$ in both cases gives
\eqref{eq:dual}. As $w_1'$, $w_2'$ are positive, we see that
$C'_0=\{\vek{0}\}$.  Since $C$ is not a one-weight code, there exists
  $\vek{y}\in R^n$ such that $\vek{0}\neq\mat{G}\vek{y}^\tp\notin
  \vek{g}_1R^\times\cup\dots\cup\vek{g}_nR^\times$ (cf.\
  Fact~\ref{fact:equidistant}). Hence both weights $w_1'$, $w_2'$ actually
  occur and the proof of Th.~\ref{thm:dual} is complete.
\begin{theorem}
  \label{thm:Gamma(C')}
  Under the assumptions of Th.~\ref{thm:dual}, the graph $\Gamma(C')$
  is strongly regular with parameters
  \begin{equation*}
    N'=\abs{C},\quad K'=\frac{n}{r},\quad
    \lambda'=\frac{2n-w_1-w_2}{r}+\frac{w_1w_2}{r^2\abs{C}},\quad
    \mu'=\frac{w_1w_2}{r^2\abs{C}}.
  \end{equation*}
  The graph $\Gamma(C')$ is trivial iff $w_1=n$ (\ie\ iff $\Gamma(C)$
  is trivial).
\end{theorem}
\begin{proof}
  Since $C'$ is a modular two-weight code with $C'_0=\{\vek{0}\}$, the
  graph $\Gamma(C')=\Gamma(C',C_1')$ is strongly regular. It remains
  to compute the paremeters of $\Gamma(C')$.

  The proof of Th.~\ref{thm:dual} shows that $\vek{y}\notin C^\perp$
  implies $\mat{M}_1\vek{y}^\tp\neq\{\vek{0}\}$. Hence $C$ is
  generated by the codewords of weight $w_1$, $C^\perp=\{\vek{y}\in
  R^n;\mat{M}_1\vek{y}^\tp=\vek{0}\}$  and
  $N'=\abs{C'}=\abs{R^n}/\abs{C^\perp}=\abs{C}$. This in turn gives
  for the frequencies $b'_1$, $b'_2$ the system of equations
  $b'_1+b'_2=\abs{C}-1$, $b'_1w'_1+b'_2w'_2=b_1\abs{C}$. Solving for
  $b'_1$ we obtain, using
$w_2'-b_1=\frac{(w_2-n)\abs{C}}{w_2-w_1}-\frac{(w_2-n)\abs{C}-w_2}{w_2-w_1}
=\frac{w_2}{w_2-w_1}$, $w_2'-w_1'=\frac{(w_2-n)\abs{C}}{w_2-w_1}
-\frac{(w_2-n-r)\abs{C}}{w_2-w_1}=\frac{r\abs{C}}{w_2-w_1}$,
  \begin{equation*}
    K'=b'_1=\frac{(w'_2-b_1)\abs{C}-w'_2}{w'_2-w'_1}
    =\frac{\frac{w_2}{w_2-w_1}\cdot\abs{C}-\frac{w_2-n}{w_2-w_1}\cdot\abs{C}}
    {\frac{r\abs{C}}{w_2-w_1}}=\frac{n}{r}
  \end{equation*}
and further 
\begin{align*}
  \mu'&=\frac{K'\left(\frac{w_1'w_2'}{n'}-w_1'\right)}{w_2'-w_1'}
  =\frac{\frac{n}{r}\left(\frac{w_1'w_2'}{b_1}-w_1'\right)}{w_2'-w_1'}
  =\frac{\frac{nw_1'}{rb_1}(w_2'-b_1)}{w_2'-w_1'}\\
  &=\frac{\frac{w_1}{r}\cdot w_2}{(w_2'-w_1')(w_2-w_1)}
  =\frac{w_1w_2}{r^2\abs{C}},\\
  \lambda'-\mu'
  &=\frac{K'\left(\frac{{w_1'}^2}{n'}-2w_1'\right)+w_2'(K'-1)}{w_2'-w_1'}
  -\frac{K'\left(\frac{w_1'w_2'}{n'}-w_1'\right)}{w_2'-w_1'}\\
  &=-\frac{K'w_1'}{n'}+K'-\frac{w_2'}{w_2'-w_1'}
  =-\frac{\frac{n}{r}\cdot\frac{b_1w_1}{n}}{b_1}+\frac{n}{r}-\frac{w_2-n}{r}
  =\frac{2n-w_1-w_2}{r},\\
  \lambda'&=(\lambda'-\mu')+\mu'
  =\frac{2n-w_1-w_2}{r}+\frac{w_1w_2}{r^2\abs{C}}.
\end{align*}

By Th.~\ref{thm:Gamma(C)}, the graph $\Gamma(C')$ is trivial iff $w_1'=n'$.
But $w_1'=\frac{b_1w_1}{n}$ and $n'=b_1$, so $\Gamma(C')$ is trivial
iff $w_1=n$.
\end{proof}
\begin{theorem}
  \label{thm:equivalence}
  Let $C\leq{}_RR^n$ be a modular linear code over a finite Frobenius
  ring $R$ generated by $\mat{G} =(\vek{g}_1|\dots|\vek{g}_n)\in
  R^{k\times n}$. Let 
  %$D=\{\mat{G}\vek{y}^\tp;\vek{y}\in R^n\}\leq R^k_R$ 
  $D\leq R^k_R$ be the right column space of $\mat{G}$.  Suppose
  $C$ has no all-zero coordinate and satisfies $C_0=\bigl\{\vek{c}\in
  C;\whom(\vek{c})=0\bigr\}=\{\vek{0}\}$. Then the following are
  equivalent:
  \begin{enumerate}[(i)]
  \item $C$ is a homogeneous two-weight code;
  \item $\Omega=\vek{g}_1R^\times\cup\dots\cup\vek{g}_n R^\times$ is a partial
    difference set in $(D,+)$ and $\Omega\cup\{\vek{0}\}$ is not a
    submodule of $R^k_R$.
  \end{enumerate}
\end{theorem}
\begin{proof}
  (i)$\oif$(ii): By Th.~\ref{thm:dual} the dual $C'$ is a modular
  two-weight code with $C'_0=\{\vek{0}\}$ and by
  Th.~\ref{thm:Gamma(C')} the graph $\Gamma(C')=\Gamma(C',C_1')$ is
  strongly regular, \ie\ the set $C_1'\subset C'$ of codewords of
  weight $w_1'$ is a partial difference set in $(C',+)$. Let
  $\mat{M}_1\in R^{b_1\times n}$ be the matrix used to define $C'$;
  cf.\ Def.~\ref{dfn:dual}. There exists $\mat{X}\in R^{b_1\times k}$
  with $\mat{XG}=\mat{M}_1$. The proof of Th.~\ref{thm:Gamma(C')}
  shows that $\mat{XG}\vek{y}^\tp=\mat{M}_1\vek{y}^\tp=\vek{0}$
  implies $\mat{G}\vek{y}^\tp=\vek{0}$. Hence
  $f(\vek{y}^\tp):=\mat{X}\vek{y}^\tp$ defines an right $R$-module
  isomorphism $f$ from $D=\{\mat{G}\vek{y}^\tp;\vek{y}\in R^n\}$ to
  $C'=\{\mat{M}_1\vek{y}^\tp;\vek{y}\in R^n\}$. Again by the proof of
  Th.~\ref{thm:Gamma(C')}, $\whom(\mat{M}_1\vek{y}^\tp)=w_1'$ iff
  $\mat{G}\vek{y}^\tp\sim\vek{g}_j$ for some $j$, \ie\ iff
  $\mat{G}\vek{y}^\tp\in\Omega$. In other words, we have
  $f(\Omega)=C_1'$. Clearly this implies that $\Omega$ is a partial
  difference set in $(D,+)$. The second assertion of (ii) follows from
  Fact~\ref{fact:equidistant}.
  
  (ii)$\oif$(i): Since
  \begin{equation*}
    \whom(\vek{x}\mat{G})
    =\sum_{\vek{g}R\in\points}\alpha(\vek{g}R)\whom(\vek{x}\cdot\vek{g})
    =r\sum_{\vek{g}\in\Omega}\whom(\vek{x}\cdot\vek{g}),
  \end{equation*}
  it suffices to show that
  $\vek{x}\to\sum_{\vek{g}\in\Omega}\whom(\vek{x}\cdot\vek{g})$ takes
  no more than two nonzero values on $R^k$.

  By \cite[Cor.~3.3]{ma94} there are $v_1,v_2\in\R$ such that for any
  nontrivial (complex) character $\lambda$ of $(D,+)$ we have
  $\lambda(\Omega)=\sum_{\vek{g}\in\Omega}\lambda(\vek{g})\in\{v_1,v_2\}$. 
  If $\chi$ is a generating character of $R$ then clearly
  $D\to\C^\times$, $\vek{g}\mapsto\chi(\vek{x}\cdot\vek{g})$ is a
  character of $(D,+)$. Hence
  \begin{align*}
    \sum_{\vek{g}\in\Omega}\whom(\vek{x}\cdot\vek{g})
    &=\sum_{\vek{g}\in\Omega}\left(1-\frac{1}{\abs{R^\times}}\sum_{u\in
        R^\times}\chi(\vek{x}\cdot\vek{g}u)\right)\\
    &=\abs{\Omega}-\frac{1}{\abs{R^\times}}\sum_{u\in
        R^\times}\sum_{\vek{g}\in\Omega}\chi(\vek{x}\cdot\vek{g}u)\\
    &=\abs{\Omega}-\sum_{\vek{g}\in\Omega}\chi(\vek{x}\cdot\vek{g})
  \end{align*}
takes only the values $0$ (if $\vek{x}\mat{G}=\vek{0}$),
$\abs{\Omega}-v_1$ or $\abs{\Omega}-v_2$. Hence
$C=\{\vek{x}\mat{G};\vek{x}\in R^k\}$ is a two weight code with nonzero
weights $w_1=r\bigl(\abs{\Omega}-v_1\bigr)$,
$w_2=r\bigl(\abs{\Omega}-v_2\bigr)$. (By Fact~\ref{fact:equidistant}
and the assumption that $\Omega\cup\{\vek{0}\}$ is not a submodule of
$R^k_R$, it cannot be a one-weight code.)
\end{proof}
\begin{remark}
  \label{rmk:equivalence}
  Under the assumptions of Th.~\ref{thm:equivalence} the set
  $\Omega\cup\{\vek{0}\}$ is a submodule of $R^k_R$ iff $C$ is a
  homogeneous one-weight code, and $D\setminus\Omega$ is a submodule
  of $R^k_R$ iff $C$ is a homogeneous two-weight code with
  $w_1=n$. 
  %This also follows Theorem~\ref{thm:dual}.
\end{remark}

%\nocite{wood99a,wood99,calderbank-kantor86,emt:oc4graphs,emt:graphs,ioana-werner97}
% \bibliographystyle{abbrv}
% \bibliography{strings,th,mathe}
% \end{document}

\def\cprime{$'$}

\end{document}